\documentclass[12pt,reqno]{amsart}

\usepackage[left=3cm,top=2cm,right=3cm,bottom=2cm]{geometry}
\usepackage{amsmath,amsfonts}
\usepackage{amssymb,amsthm}
\usepackage{graphicx}
\usepackage{graphics}
\usepackage[latin1]{inputenc}
\usepackage[T1]{fontenc}
\usepackage[english]{babel}
\usepackage{geometry}
\usepackage{tikz}
\usetikzlibrary{backgrounds,arrows,shapes,trees} 
\title{No two starlikes have equal index}
\author[E. R. Oliveira]{Elismar R. Oliveira}\email{\tt elismar.oliveira@ufrgs.br}
\address{UFRGS - Universidade Federal do Rio Grande do Sul, Instituto de Matem\'atica, Porto Alegre, Brazil}
\author[D. Stevanovi\'c]{Dragan Stevanovi\'c}\email{dragance106@yahoo.com}
\author[V. Trevisan]{Vilmar Trevisan} \email{\tt trevisan@mat.ufrgs.br}
\address{UFRGS - Universidade Federal do Rio Grande do Sul, Instituto de Matem\'atica, Porto Alegre, Brazil}

\newtheorem{exemplo}{Example}

\newtheorem{thm}{Theorem}

\newtheorem{prop}{Proposition}
\newtheorem{lema}{Lemma}

\begin{document}
\begin{abstract}
The index of a graph is the largest eigenvalue of its adjacency matrix. A starlike is a tree having a unique vertex of degree $r>2$. We show how to order the starlike trees with $n>3$ by their indices. In particular, the index of starlike trees are all distinct.
\end{abstract}

\maketitle


\section{Introduction}

A starlike is a tree with a unique vertex of degree $r\geq 3$. It may be seen as a graph having a central vertex attached to $r$ paths with $y_1,\ldots, y_r$, respectively (see Figure \ref{slike} for an example). In spite of its apparent simplicity, import spectral properties of this class of graphs have been derived in many  studies.

Recall that the spectrum of a graph $G$ having $n$ vertices is the (multi)set of the eigenvalues of its adjacency matrix. We traditionally order them so that $$ \lambda_1 \geq \lambda_2 \geq \cdots \geq \lambda_n.$$
The index of the graph $G$ is $\lambda_1$, the largest eigenvalue of $G$, which is positive if $G$ is not the empty graph.

When defining starlike tree in 1979, Watanabe and Schwenk \cite{watanabe1979} studied the integrality of such trees, that is, they characterized starlike trees for which all the eigenvalues are integral values. In \cite{LEPOVIC2002}, Lepovi\'c and Gutman showed that starlikes are characterized by their spectrum, meaning that no two (non isomorphic) starlikes have the same spectrum.
In \cite{Patuzzi}, Patuzzi, de Freitas and Del-Vecchio studied the integrality of the index of starlike trees.


In this note, we show that two non isomorphic starlike trees having $n>3$ vertices have different indices. In fact our result is stronger than this: given $n>3$ we can order all non isomorphic starlike trees with $n$ vertices by their indices.

To explain how our result is proven, we set the following notation.
For a starlike tree with $n\geq 4$ vertices and $r\geq 3$ paths $P_{y_1},\ldots , P_{y_r}$, each one with $y_i$ vertices, the usual notation is $S(y_1,\ldots,y_r)$. Notice that by setting $y_1\leq y_2 \leq \cdots \leq y_r$, with $y_1+\cdots + y_r=n-1$, each distinct set of numbers $[y_1,y_2,\ldots ,y_r]$ gives a non isomorphic starlike tree with $n$ vertices. But recall that each such an $r$-tuple is a \emph{partition of $n-1$ in $r$ parts}. Henceforth, we are going to refer the partition $[y_1,y_2,\ldots ,y_r]$ of $n-1$ of $r$ parts and the starlike $S(y_1,\ldots,y_r)$ with $r$ paths indistinctly. Figure \ref{slike} shows the starlike $[1,2,3,3,5]$ having $n=14$ vertices and 5 paths.

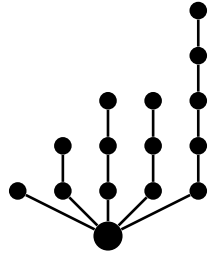
\begin{figure}[h]
\begin{center}
  \begin{tikzpicture}[scale=.6, transform shape]
													
	\tikzstyle{every node} = [circle, fill=black]
	\node(c) at (0,0){c}; 
	
	\node(r11) at (-2,1){};
	\node(r21) at (-1,1){};
    \node(r01) at (0,1){};
	\node(r31) at (1,1){};
	\node(r41) at (2,1){};
		
	
	\node(r22) at (-1,2){};
	\node(r32) at (1,2){};
    \node(r02) at (0,2){};   
	\node(r42) at (2,2){};
	
	
	\node(r04) at (0,3){};
	\node(r34) at (1,3){};
	\node(r44) at (2,3){};
	
	
    \node(r45) at (2,4){};	
	\node (r55) at (2,5){};
								
	\foreach \from/\to in {c/r11,c/r21,c/r01,c/r31,c/r41,r21/r22,r01/r02,r02/r04,r31/r32,r32/r34,r41/r42,r42/r44,r44/r45,r45/r55} 
	\draw[line width=1] (\from) -- (\to); 
  \end{tikzpicture}
\end{center}
\caption{\label{slike}The starlike $[1,2,3,3,5]$.}
\end{figure}

Seeing the starlikes with $n$ vertices as the set of partitions, allows one to order them in \emph{lexicographical} order (see next section for definitions). What we actually show is that the order of the indices of the starlikes has the same lexicographical order of the partitions.

The paper is organized as follows. In section \ref{sec:tool} we explain the main tool used and show how to compare the indices of two starlike trees. Next, in section \ref{sec:order}, we define and study the lexicographical order of partitions. In particular, we show how to obtain the successor of a given partition. Finally, in section \ref{sec:mainresult} we prove the main result.

\section{The main tool}\label{sec:tool}

In this section, we describe our main tool to prove the result. The starting poing is the algorithm due to Jacobs and Trevisan~\cite{Jac11} that can be used to estimate the eigenvalues of a given tree.  It counts the number of eigenvalues of the adjacency matrix of a tree $T$ lying in any real interval. The algorithm is based on the diagonalization of the matrix $A(T)+ \alpha I$, where $A(T)$ is the adjacency matrix of $T$ and $\alpha$ is a real number. One of the main features of this algorithm is that it can be executed directly on the tree, so that the adjacency matrix is not needed explicitly. The algorithm is given in Figure \ref{treealgo}.

\begin{figure}[h]
{\tt
\begin{tabbing}
aaa\=aaa\=aaa\=aaa\=aaa\=aaa\=aaa\=aaa\= \kill
     \> Input: tree $T$, scalar $\alpha$ \\
     \> Output: diagonal matrix $D$ congruent to $A(T) + \alpha I$ \\
     \> \\
     \>  Algorithm $\mbox{Diagonalize}(T, \alpha)$ \\
     \> \> initialize $d(v) := \alpha$, for all vertices $v$ \\
     \> \> order vertices bottom up \\
     \> \> {\bf for } $k = 1$ to $n$ \\
     \> \> \> {\bf if} $v_k$ is a leaf {\bf then} continue \\
     \> \> \> {\bf else if} $d(c) \neq 0$ for all children $c$ of $v_k$ {\bf then} \\
     \> \> \>  \>   $d(v_k) := d(v_k) - \sum \frac{1}{d(c)}$, summing over all children of $v_k$ \\
     \> \> \> {\bf else } \\
     \> \> \> \> select one child $v_j$ of $v_k$ for which $d(v_j) = 0$  \\
     \> \> \> \> $d(v_k)  := -\frac{1}{2}$ \\
     \> \> \> \> $d(v_j)  :=  2$ \\
     \> \> \> \> if $v_k$ has a parent $v_l$, remove the edge $v_k v_l$. \\
     \> \>  {\bf end loop} \\
\end{tabbing}
}
\caption{\label{treealgo} Diagonalizing $A + \alpha I$.}
\end{figure}

It is worth noticing that the diagonal elements of the output matrix correspond precisely to the values $a(v)$ on each node $v$ of the tree. The following is the result we are going to use.

\begin{lema}[Jacobs and Trevisan~\cite{Jac11}]
\label{mainA}
Let $T$ be a tree and let $D$ be the diagonal matrix produced by the algorithm \verb+Diagonalize+$(T, -\alpha)$. The following assertions hold.
\begin{itemize}
\item[(a)]
The number of positive entries in $D$ is the number of eigenvalues
of $T$ that are greater than $\alpha$.
\item[(b)]
The number of negative entries in $D$ is the number of eigenvalues
of $T$ that are smaller than~$\alpha$.
\item[(c)]
If there are $j$ zero entries in $D$, then $\alpha$ is an eigenvalue
of $T$ with multiplicity~$j$.
\end{itemize}
\end{lema}

To illustrate how the algorithm performs, we look at an example. Consider the starlike tree $T=S(1,3,3)$ with 8 vertices and let us apply the algorithm Diagonalize($T$,-1), where the root is chosen to be the only of vertex of degree 3. We initialize all vertices with -1 (left of Figure~\ref{FigExample}) . We process the vertices from the leaves towards the root. The middle vertices of each of the paths $P_3$ become 0, after processing them. In order to process the last vertices of the $P_3$'s, we look at their children, that have value 0. Hence the algorithm requires to assign the value $-\frac{1}{2}$ for them, $2$ for the original vertices that were zero, and eliminate the edge between them and its parent (in this case, the root), see center Figure \ref{FigExample}. Now the only remaining vertex to process is the root, that currently has a single children, whose value is -1. Processing it, results in a value 0. The vertices have the final values $(-1)$ 3 times, $-\frac{1}{2}$ twice, 0 once and 2 twice. This shows there are 5 eigenvalues smaller than 1, 1 is an eigenvalue  of multiplicity 1 and there are 2 eigenvalues larger than 1. The right of Figure \ref{FigExample} shows the final values.
\begin{figure}[h]
\begin{center}
  \begin{tikzpicture}[scale=.6, transform shape]
													
	\tikzstyle{every node} = [circle, fill=gray]
	\node[label=left:-1](c) at (0,0){}; 
	
	\node[label=left:-1](r21) at (-1,1){};
    \node[label=left:-1](r01) at (0,1){};
	\node[label=right:-1](r31) at (1,1){};
		
	
	\node[label=right:-1](r32) at (1,2){};
    \node[label=left:-1](r02) at (0,2){};
	
	
	\node[label=left:-1](r04) at (0,3){};
	\node[label=right:-1](r34) at (1,3){};

								
	\foreach \from/\to in
{c/r21,c/r01,c/r31,r31/r32,r32/r34,r01/r02,r02/r04} 
	\draw[line width=1] (\from) -- (\to); 

	\node[label=left:-1](c) at (4,0){}; 
	
	\node[label=left:-1](r21) at (3,1){};
    \node[label=left:-1](r01) at (4,1){};
	\node[label=right:-1](r31) at (5,1){};
		
	
	\node[label=right:0](r32) at (5,2){};
    \node[label=left:0](r02) at (4,2){};
	
	
	\node[label=left:-1](r04) at (4,3){};
	\node[label=right:-1](r34) at (5,3){};
								
	\foreach \from/\to in
{c/r21,c/r01,c/r31,r31/r32,r32/r34,r01/r02,r02/r04} 
	\draw[line width=1] (\from) -- (\to); 

	\node[label=left:-1](c) at (8,0){}; 
	
	\node[label=left:-1](r21) at (7,1){};
    \node[label=left:\tiny{-1/2}](r01) at (8,1){};
	\node[label=right:\small{-1/2}](r31) at (9,1){};
		
	
	\node[label=right:2](r32) at (9,2){};
    \node[label=left:2](r02) at (8,2){};
	
	
	\node[label=left:-1](r04) at (8,3){};
	\node[label=right:-1](r34) at (9,3){};
								
	\foreach \from/\to in
{c/r21,r31/r32,r32/r34,r01/r02,r02/r04} 
	\draw[line width=1] (\from) -- (\to); 

  \end{tikzpicture}

\end{center}
\caption{\label{FigExample}The starlike $S[1,3,3]$.}
\end{figure}
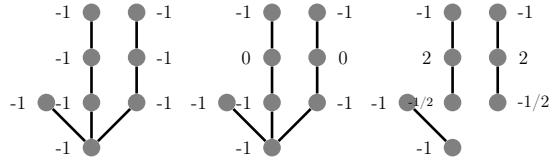 

Consider now a starlike tree $T=S(y_1,\ldots, y_r)=[y_1,\ldots, y_r]$. Let $\lambda=-\lambda_1(T)$, where $\lambda_1(T)$ is the index of $T$, i.e. the largest eigenvalue of $T$. After applying the algorithm Diagonalize($T, \lambda$), Theorem \ref{mainA} tells us that all values will be negative, except for one value that must be 0. We observe that the value 0 can only happen while processing the last vertex,  the root, otherwise there would be positive values.

Consider a path $P_{y_{j}}$ of $T$. The sequence $a_k$ for $k = 1,\ldots y_j-1$ given by the recurrence equation $$a_1=\lambda \text{ and } a_{k+1}= a_1 - \frac{1}{a_{k}}=\lambda - \frac{1}{a_{k}},$$ define the values appearing in each vertex of the path $P_{y_j}$. After processing all paths of $T$, we process the root. The value at root is
$$\lambda - \frac{1}{a_{y_1}} - \cdots - \frac{1}{a_{y_r}}.$$
We know that this values has to be 0, hence we have the following equation for $\lambda$.
\begin{equation}\label{ais}
  \lambda = \frac{1}{a_{m_1}} + \cdots + \frac{1}{a_{m_r}}.
\end{equation}

We reformulate this as the following result.
\begin{thm}\label{formulacao}
 Let $0 < y_1\leq y_2 \leq \cdots \leq y_r$ and $r \geq 3$ be positive integers. Let $T$ be the starlike $S(y_1,\ldots,y_r)$. For a given $\lambda$, define the recurrence relation
$b_{k}=\frac{1}{a_{k}}$ for $k=1,2, \ldots $ satisfying
\begin{equation}\label{bis1}
b_1=\frac{1}{\lambda} \text{ and } b_{k+1}= \frac{1}{\lambda - b_{k}}.
\end{equation}
If $\lambda=-\lambda_1(T)$, where $\lambda_1(T)$ is the index of $T$, then
\begin{equation}\label{bis}
 b_{y_1}+ \cdots + b_{y_r}=\lambda.
\end{equation}
\end{thm}

This next lemma defines properties of the recurrence relation $b_i$ given by (\ref{bis1}) that are essential to our proofs.

\begin{lema} \label{b_n properties} Let $T=[y_1,\ldots,y_r]$ be a starlike. Consider the sequences $\{b_i\}_{i=1,2,\ldots,y_r}$ associated with $\lambda=-\lambda_1(T)$. Then
\begin{itemize}
 \item[(a)] $b_i <0$, for $i=1,\ldots,y_r$, and is decreasing;
 \item[(b)] For any $1 \leq \theta \leq y_r$ we have $\displaystyle \lambda > b_{\theta}+\frac{1}{b_{\theta}}$.
\end{itemize}
\end{lema}
\begin{proof}
(a) We first observe that the sequence $a_i<0,~ i=1,\ldots y_r$. This is so by Theorem \ref{mainA}. Notice that $0 > a_1=\lambda = -\lambda_1(T)$ is the negative of the index of $T$. The number of $i$'s to consider is the number of vertices of the largest path, that is $y_r$. If $y_r\geq 2$, then $a_2-a_1=a_1-\frac{1}{a_1}-a_1=-\frac{1}{a_1}=-\frac{1}{\lambda} >0$. Suppose, by induction, that $a_{k}-a_{k-1}>0$ for some $k$. Then
$$ a_{k+1}-{a_k} = a_1 - \frac{1}{a_k} - a_1 + \frac{1}{a_{k-1}}=
 \frac{1}{a_{k-1}}- \frac{1}{a_k} =\frac{a_k-a_{k-1}}{a_k a_{k-1}} >0.
$$
This shows that the sequence $a_i, ~i = 1 \ldots, y_r-1$ is increasing. Now, since $b_1=\frac{1}{\lambda}<0$ and for $i = 1,\ldots, y_r$,
$$b_{i+1} - b_{i}= \frac{1}{a_{i+1}}  -  \frac{1}{a_{i}}= \frac{a_{i} - a_{i+1}}{a_{i}a_{i+1}} <0$$ because $a_i<0$ is increasing.

(b) From (a) we obtain  $b_{\theta}>b_{\theta+1}=\frac{1}{\lambda -b_{\theta}}$. Hence
$$b_{\theta}>\frac{1}{\lambda -b_{\theta}} \Rightarrow b_{\theta}(\lambda -b_{\theta})<1$$
$$\lambda -b_{\theta}>\frac{1}{b_{\theta}} \Rightarrow\lambda > b_{\theta}+\frac{1}{b_{\theta}}.$$
\end{proof}

Now given a starlike $T=[y_1,\ldots,y_r]$ with index $\lambda_1(T)$, the equation (\ref{bis}) may be read as $b_{y_1}+ \cdots + b_{y_r}=-\lambda_1(T)$. We want to compare the index of $T$ with another starlike, say $T^\prime=[z_1,\ldots,z_s]$. Let us assume that the largest path of $T^\prime$ is no larger than the largest path of $T$, equivalently,that
$$y_r \geq z_s.$$
Consider applying the algorithm Diagonalize to $T^\prime$ with the same $\lambda=-\lambda_1(T)$.
Notice that the values in all vertices (but the root) are negative by Lemma \ref{b_n properties}, because these are the same $b_i$ associated with $T$, for some natural number $i$. Let us then consider the last value, computed at the root, given by
$$L=\lambda - (b_{z_1}+ \cdots + b_{z_s}).$$
$L$ is either positive, negative or zero. In the context of Lemma \ref{mainA}, we conclude the following. If $L<0$, then, since all the values are negative, all the eigenvalues of $T^\prime$ are smaller than $\lambda_1(T)$, including its index. If $L=0$, then $\lambda_1(T) = \lambda_1(T^\prime)$. And, finally, if $L > 0$, we have that $n-1$ eigenvalues of $T^\prime$ are smaller than $\lambda_1(T)$ and exactly one (the index) is larger than $\lambda_1(T)$.

We state this for reference, since this is the main tool to prove our result.

\begin{thm} \label{maintool} Let $T=[y_1,\ldots,y_r]$ and $T^\prime=[z_1,\ldots,z_s]$ be two non isomorphic starlikes with $n$ vertices. Additionally, assume $r,s >2$ and $z_s \leq y_r$. If
$$  (b_{y_1}+ \cdots + b_{y_r})-(b_{z_1}+ \cdots + b_{z_s})  >0,$$
then $$\lambda_1(T) <  \lambda_1(T^\prime),$$
that is, the index of $T^\prime$ is larger than the index of $T$.
\end{thm}

\section{Ordering partitions}\label{sec:order}

Ordering partitions lexicographically is somewhat obvious. For our purposes, however, we need to understand exactly how to change one partition to find the next one. The main goal of this section is to characterize the types of changes that make one partition to be the successor of a given partition.

We define, for each $n\geq 4$ and $r \geq 3$ the set of ordered natural numbers
$$\Omega_{n-1,r}:=\left\{[y_1, \ldots, y_r] \, | \, \sum_{k=1}^{r} y_k =n-1 \text{ and } y_1 \leq \cdots\leq y_r \right\},$$
as the set of partitions of $n-1$ with $r$ parts. Each partition $[y_1, \ldots, y_r]$ corresponds to a starlike tree with $n$ vertices and $r$ paths each one with $y_i$ vertices.

We introduce the lexicographical order $\prec$ in  $\Omega_{n-1,r}$ by
$$[y_1,\ldots, y_r] \prec [z_1, \ldots, z_r] \Leftrightarrow \{y_i=z_i, \; 1\leq i\leq s-1 \text{ and } y_s < z_s \},$$
where $[y_1, \ldots, y_r] \neq [z_1, \ldots, z_r]$.

We consider the set $\Omega_{n-1}:=\Omega_{n-1,3} \cup \Omega_{n-1,4} \cup \cdots \cup \Omega_{n-1,n-1}$ that corresponds to the set of all starlike trees with $n$ vertices, and extend the lexicographical order by assuming that \begin{equation}\label{jump}
  [m_1, \ldots, m_r] \prec [y_1, \ldots, y_{r+1}],
\end{equation}
for any $[m_1, \ldots, m_r] \in \Omega_{n-1,r}$ and $[y_1, \ldots, y_{r+1}] \in \Omega_{n-1,r+1}$ then $(\Omega_{n-1}, \prec)$ is a totally ordered set.
For instance,
{\footnotesize
$$ \Omega_{7}= \{[1, 1, 5],[1, 2, 4],[1, 3, 3],[2, 2, 3],[1, 1, 1, 4],$$ $$[1, 1, 2, 3],[1, 2, 2, 2],[1, 1, 1, 1, 3],[1, 1, 1, 2, 2],[1, 1, 1, 1, 1, 2],[1, 1, 1, 1, 1, 1, 1]\}.$$}

Given a set $\Omega_{n-1}$, we want to describe its structure. In particular, we want to identify consecutive elements, and more precisely, we want to define general operations that transform a partition into a new partition having no other partition between them.

We say that a partition $[m_1, \ldots, m_t] \in \Omega_{n-1}$ \emph{covers} $[y_1, \ldots, y_{s}] \in \Omega_{n-1}$ if they are consecutive, that is if $[y_1,\ldots, y_{s}]\prec [m_1, \ldots, m_t]$ and there is no partition between them. An operation that transforms  $[y_1, \ldots, y_{s}]$ into the partition that covers it will be called a \emph{covering}. A covering is a transformation that when applied to a partition $\gamma$, determines the \emph{successor} of $\gamma$.

Giving the observation in (\ref{jump}), we see that all the partitions in $\Omega_{n-1,s}$ are smaller than any partition in $\Omega_{n-1,t}$ for all $ 3\leq s < t \leq n-1$. Hence, we are going to investigate more closely the subset $\Omega_{n-1,r}$ for a given $r\geq 3$.

Let us consider, for the sake of an example, the subset

{\footnotesize$$\Omega_{11,3}=\left\{{\color{red} [1,1,9],[1,2,8],[1,3,7], [1,4,6],[1,5,5]},{\color{black}[2,2,7], [2,3,6], [2,4,5]}, {\color{blue}[3,3,5], [3,4,4]} \right\}.$$}

Clearly, the \emph{smallest} partition  with respect to the lexicographical order in $\Omega_{n-1,r}$ is given by
$$[y_1, \ldots, y_r]:=[1,\ldots,1, (n-1)-(r-1)]= [1,\ldots,1, n-r]$$
and it is well defined if $n-1 \geq r$. The \emph{largest} partition will be the \emph{balanced} partition, the one having all values equal to $m=\lfloor \frac{n-1}{r} \rfloor$, and if $l=(n-1)-r*m>0$, then $l$ values will be $m+1$. Its configured as $$[\underbrace{m,\ldots,m}_{r-l},\underbrace{m+1,\ldots,m+1}_{l}].$$

In order to understand how to determine a successive partition, we define the following \emph{$\alpha$ operation} in $\Omega_{n-1,r}$.
Fixed $[y_1,\ldots , y_r] \in \Omega_{n-1,r}$  and $1\leq i < j \leq r$ we define $$\alpha_{i,j}([y_1, \ldots, y_r])= [y_1, \ldots,y_{i}+1, \ldots, y_{j}-1, ..., y_r],$$
if $y_{j} \geq 2$ and $y_{i}+1 \leq y_{i+1}$. For example $\alpha_{1,2}([2,4,5])=[3,3,5]$ and $\alpha_{1,3}([2,4,5])=[3,4,4]$. We notice that $$[y_1, \ldots, y_r] \prec \alpha_{i,j}([y_1, \ldots, y_r])$$
that is, an $\alpha$ operation  preserves order, it produces a larger partition. However, as the second example shows, an $\alpha$ transformation, in general, does not produce a successor, that it, is not always a covering.

Starting with the smallest partition, we construct a \emph{maximal class} in $\Omega_{n-1,r}$ by successively applying $\alpha_{r-1, r}$  to $\gamma=[1,\ldots,1, n-r]$, obtaining $[1,\ldots,z_{r-1},z_r]$ until $z_{r-1}=z_{r}$ or $z_{r-1}+1=z_{r}$:
$$\begin{array}{rl}
    \alpha_{r-1, r}^{0}(\gamma)  =[1,\ldots,1, n-r] & \prec\alpha_{r-1, r}^{1}(\gamma) =[1,\ldots,2, n-r-1])\prec  \\
    & \prec \cdots  \prec \alpha_{r-1, r}^{t}(\gamma)=[1, \ldots,1 ,z_{r-1}, z_r].
  \end{array}
$$

Let us nominate the set of partitions produced by iterating $\alpha_{r-1, r}$, the \emph{orbit} of $\gamma=[1,\ldots,1, n-r]$.
We observe that this is a finite procedure and, most importantly, each transformation  $\alpha_{r-1,r}^i(\gamma)$ is a covering. That is, for each $i=1,\ldots,t$, $\alpha_{r-1, r}^{i}([1,\ldots,1, n-r])$ produces a successive partition.

Having determined the largest partition of the initial maximal class, we find the smallest partition of the next class, simply replacing $z_{r-2}=1$ by $z_{r-2}=2$:
$$[1, \ldots,1 ,z_{r-1}, z_r] \prec [1, \ldots, 2,2, (n-1)-(r-2)\cdot 1 - 2\cdot 2].$$
To find consecutive elements, we iterate the application of $\alpha_{r-1, r}$ producing the next maximal class, and so on.

Back to our example $\Omega_{11,3}$ we have 3 maximal classes and the ordering  {\footnotesize$$\underline{[1,1,9] \prec [1,2,8] \prec [1,3,7] \prec [1,4,6] \prec [1,5,5]} \prec $$ $$\prec \underline{[2,2,7] \prec  [2,3,6] \prec  [2,4,5]} \prec \underline{[3,3,5] \prec [3,4,4]}.$$}

The \emph{minimum} partition of each maximal class is given by
$$[m_1,\ldots,m_t,k, \ldots, k, (n-1)-(m_1+\cdots+m_t)- s k].$$

The conclusion is that each maximal class is the orbit of the minimal partition. The minimal partitions of $\Omega_{11,3}$ are  $[1,1,9]\prec [2,2,7]\prec [3,3,5]$.

The procedure above described obviously generates an ordered union of maximal classes, that obviously is contained in $\Omega_{n-1,r}$. To see the equality we will next analyze the coverings in detail.

\subsection{Coverings}

A covering is a transform that when applied to a partition produces its successor. Which means that this pair of partitions are consecutive, having no partition between them in the lexicographical order and the second one is bigger.

More precisely, we will say that $[z_1, \ldots, z_s]$ covers  $[y_1, \ldots, y_r] $ and denote $$[y_1, \ldots, y_r] \rightarrow [z_1, \ldots, z_s]$$ if
\begin{enumerate}
  \item[(a)] $[y_1, \ldots, y_r] \prec [z_1, \ldots, z_s]$ (in particular $s \geq r$);
  \item[(b)]There is no $[x_1, \ldots, x_t]$ such that  $[y_1, \ldots, y_r] \prec [x_1, \ldots, x_t] \prec [z_1, \ldots, z_s]$.\\
\end{enumerate}

To describe all the coverings in $\Omega_{n-1}$  we classify them according to the previous discussion. There are, hence, three types of coverings: (I) the one going from the largest partition of $\Omega_{n-1,r}$ to the smallest partition $\Omega_{n-1,r+1}$; (II) the ones within a maximal class and (III) the one going from the largest partition of a maximal class to the smallest partition of the next maximal class. The next result characterizes these coverings.

\begin{thm} \label{transitions}
The coverings in $\Omega_{n-1}$ are characterized by\\

\noindent {\bf Type I:} $[y_1, \ldots, y_r] \rightarrow [z_1, \ldots, z_{r+1}]$.\\
Here $[y_1, \ldots, y_r]= [\underbrace{m,\ldots,m}_{r-l},\underbrace{m+1,\ldots,m+1}_{l}]$ is the balanced partition in $\Omega_{n-1,r}$ and $[z_1, \ldots, z_{r+1}]=[1,\ldots,1,(n-1-r)$ is the smallest partition of $\Omega_{n-1,r+1}$.\\

\noindent \textbf{Type II:} $[y_1, \ldots, y_r] \rightarrow  \alpha_{r-1, r}([y_1, \ldots, y_r])$.\\ In  particular $[y_1, \ldots, y_r] \rightarrow  \alpha_{r-1, r}([y_1, \ldots, y_r])$ implies that every maximal class is composed by consecutive partitions.\\

\noindent \textbf{Type III:} If the covering is not Type I or Type II then
$$[y_1, \ldots, y_r]\rightarrow [z_1, \ldots, z_r]$$
implies that \begin{itemize}
\item[(i)] either $[z_1, \ldots, z_r]=  \alpha_{t+1, r}([y_1, \ldots, y_r])$ or
\item[(ii)]$[y_1, \ldots, y_r]= [m_1,\ldots,m_t, k,\ldots,\xi]$,\\
       $[z_1,\ldots, z_r]= [m_1,\ldots,m_t, k+1,\ldots,\theta]$, with $\theta \geq \xi$.
\end{itemize}
\end{thm}
\begin{proof}
By construction, Type I coverings will be only between the bigger partition in $\Omega_{n-1,r}$ and the initial partition of $\Omega_{n-1,r+1}$. The fact that this represents a covering is obvious.

Type II transitions occurs by performing an $\alpha$ transformation. We already see that $\alpha$ preserves order thus we need to prove that $[y_1, \ldots, y_r] $ and $\alpha_{r-1, r}([y_1, \ldots, y_r])$ are consecutive. Suppose that
$$[y_1, \ldots, y_r]  \prec [x_1, \ldots, x_r]   \prec  \alpha_{r-1, r}([y_1, \ldots, y_r])= [z_1, \ldots, z_s] $$
then $y_k=x_k= z_k$ for $k<r-1$ and $y_{r-1} \leq  x_{r-1} \leq z_{r-1}=y_{r-1} +1$. As the sum must remain equal to $n-1$ we conclude that $[x_1, \ldots, x_r]  =[y_1, \ldots, y_r] $ or $[x_1, \ldots, x_r]=\alpha_{r-1, r}([y_1, \ldots, y_r])$.

Type III transitions are more complex. We separate the proof in the next two lemmas.
\end{proof}

\begin{lema} \label{interm equil}
  If $[m_1,\ldots,m_t, k, y_{t+2}, \ldots, y_r] \rightarrow [m_1,\ldots,m_t, k+1, z_{t+2}, \ldots, z_s]=[z_1, \ldots, z_s]$ then $y_r - y_{t+2} \leq 1$. That is, $y_{t+2} = k+c$ with $c \geq 0$ and $y_{r} = k+c+\delta$ where $\delta=0$ or $\delta=1$.
\end{lema}

\begin{proof}
Otherwise, if $y_r \geq  k+c+2$ we can insert a configuration $[x_1, \ldots, x_r]$ between $[y_1, \ldots, y_r]$ and $ [z_1, \ldots, z_r]$ :\\\\
$[y_1, \ldots, y_r]= [m_1,\ldots,m_t, k, k+c,\ldots, k+c+1, \ldots, k+c+2,\ldots,\xi] \prec $\\\\
$[x_1, \ldots, x_r]= \alpha_{i, j}([m_1,\ldots,m_t, k,  k+c,\ldots, k+c+1, \ldots, k+c+2,\ldots,\xi])\prec $\\\\
$[z_1, \ldots, z_r]= [m_1,\ldots,m_t, k+1, k+1, \ldots,  k+1, \theta]$,\\\\
where $i$ is the largest integer with $y_i= k+c$ and $j$ is the smallest integer with $y_i \geq  k+c+2$.
\end{proof}

\begin{lema}
  The statement for Type III covering holds.
\end{lema}
\begin{proof}
From Lemma~\ref{interm equil} we can assume that there is $c\geq 0$ and $\delta \in \{0,1\}$ such that\\
$[y_1, \ldots, y_r]= [m_1,\ldots,m_t, k, \underbrace{k+c,\ldots, k+c}_{a}, \underbrace{k+c+\delta,\ldots,\xi}_{b}]$\\
where we have  $\xi=k+c+\delta$ and the number of entries equal to $k+c$ is $a\geq 1$, the number of entries equal to $k+c+\delta$ is $b\geq 1$ and $a+b \geq 2$ because it is not a Type II.

As $[z_1, \ldots, z_r]= [m_1,\ldots,m_t, k+1, k+1, \ldots,  \theta]$ must have the same sum, we conclude that
$$k+ a(k+c) + b (k+c+\delta) = (a+b) (k+1) + \theta$$
or $$\theta=  \xi -1 + (a+b-1)(c-1) +(b-1)\delta.$$

We will analyze all the different possibilities:
\begin{enumerate}
\item  $\delta=0$:
\begin{enumerate}
\item  $c=0$: In this case we have\\
$[y_1, \ldots, y_r]= [m_1,\ldots,m_t, k, k, ,\ldots,k, k]$\\
$[z_1, \ldots, z_r]= [m_1,\ldots,m_t, k+1, k+1, \ldots,  \theta]$
a contradiction because  the sum of the second exceeds the first one.
\item $c=1$: In this case we have\\
$[y_1, \ldots, y_r]= [m_1,\ldots,m_t, k, k+1, ,\ldots,k+1, k+1]$\\
$[z_1, \ldots, z_r]= [m_1,\ldots,m_t, k+1, k+1, \ldots , k+1,\theta]$
a contradiction again because we must have $\theta=k< k+1$.
\item $c>1$: In this case we have\\
 $\theta=  \xi -1 + (a+b-1)(c-1) +(b-1)\delta= \xi -1 + (a+b-1)(c-1) \geq  \xi,$\\
because $a+b \geq 2$ and $c-1 \geq 1$.
\end{enumerate}
\item $\delta=1$:
\begin{enumerate}
\item $c=0$: In this case we have\\
$[y_1, \ldots, y_r]= [m_1,\ldots,m_t, k, k,\ldots, k,k+1,\ldots, k+1,\xi]$\\
$[z_1, \ldots, z_r]= [m_1,\ldots,m_t, k+1, k+1, \ldots,  \theta]$\\
so $\theta \geq  \xi=k+1$.
\item $c=1$: In this case we have\\
$[y_1, \ldots, y_r]= [m_1,\ldots,m_t, k, k+1,\ldots, k+1,k+2,\ldots, k+2,\xi]$\\
$[z_1, \ldots, z_r]= [m_1,\ldots,m_t, k+1, k+1, \ldots , k+1,\theta]$.\\
In this situation we can have  $\theta=k+1< k+2=\xi$, which implies that $b=1$ (otherwise the sum will be different), that is\\
$[z_1, \ldots, z_r]= [m_1,\ldots,m_t, k+1, k+1, \ldots , k+1,k+1]=$\\ $= \alpha_{t+1, r}([m_1,\ldots,m_t, k, k+1,\ldots, k+1, \ldots, k+1,k+2])=$ \\ $=~\alpha_{t+1, r}([y_1, \ldots, y_r])$, otherwise, if  $\theta > k+1$ then $\theta \geq k+2=\xi$.\\
\item $c>1$: In this case we have\\
 $\theta=  \xi -1 + (a+b-1)(c-1) +(b-1)\delta= \xi -1 + (a+b-1)(c-1) + (b-1)\delta \geq  \xi,$\\
because $a+b \geq 2$, $b\geq 1$ and $c-1 \geq 1$.
\end{enumerate}
\end{enumerate}
The conclusion is that $\theta \geq \xi$ or $[z_1, \ldots, z_r]=  \alpha_{t+1, r}([y_1, \ldots, y_r])$.
\end{proof}

\section{Ordering starlikes by their indices}\label{sec:mainresult}

In this section we prove our main result, stated as the following theorem.
\begin{thm}\label{mainresult}
Let $n\geq4$ be an integer. Let $\mathcal{S}_n$ be the set of the starlike trees with $n$ vertices and $\Omega_{n-1}$ the set of all partitions $[y_1,\ldots,y_r]$ of $n-1$, with $3\leq r \leq n-1$. The order of the indices in $\mathcal{S}_n$ is the lexicographical order in $\Omega_{n-1}$. In particular, the indices of any two starlikes in $\mathcal{S}_n$ are distinct.
\end{thm}

Recall that in section \ref{sec:order}, we discussed in detail how to order lexicographically the partitions in $\Omega_{n-1}$. More precisely, we have shown how to generate a successor partition for any partition of $\Omega_{n-1}$.  In order to achieve the monotonicity of the indices, it is sufficient to show that the index of a starlike tree given by the successor partition is larger. We are going to use the result of Theorem \ref{maintool} to compare the indices of the two starlikes. We will do this in three steps according the covering types defined above.
\begin{itemize}
  \item Monotonicity of Type I coverings. We show that the index of the starlike having largest index in $\Omega_{n-1,r}$ is smaller than the index of the starlike having smallest index in $\Omega_{n-1,r+1}$, for $r=3,\ldots,n-2$. This is done in Proposition~\ref{mono_r_r+1}.
  \item Monotonicity of Type II coverings. We next show the monotonicity of the indices in each maximal class of $\Omega_{n-1,r}$, for a fixed $ 3\geq r n-1$. This means that the operation $\alpha_{r-1,r}$ increases the index. This is going to be done in Proposition~\ref{mono_alpha}.
  \item Monotonicity of Type III coverings. Here we show that the index of the starlike having largest index in a maximal class of $\Omega_{n-1,r}$ is smaller than the index of the tree having smallest index of the next maximal class. This is done in Proposition~\ref{mon_complex}.
\end{itemize}

The following lemmas are going to be necessary to prove Proposition \ref{mono_r_r+1}.
A well known result is the following.

\begin{lema}\label{sub}
  Let $G$ be a connected graph. If $H$ is a proper subgraph of $G$ then $\lambda_1(H) <\lambda_1(G)$.
\end{lema}

The following bound is due to Lepovi\'c and Gutman \cite{Lepovic2001} (see also \cite{Patuzzi}).
\begin{lema}\label{cotamagica}
Let $T$ be a starlike with $r$ paths. Then the index of $\lambda_1(T)$ satisfies
$$\lambda_1(T) < \frac{r}{\sqrt{r-1}}.$$
\end{lema}

\begin{prop}[Monotonicity between $\Omega_{n-1,r}$ and $\Omega_{n-1,r+1}$] \label{mono_r_r+1}
Let $n>3$ and $2< r < n-1$ be integers. Let $T$ be the starlike given by the largest partition $[y_1, \ldots, y_r] \in \Omega_{n-1,r}$ and let $T^\prime$ be the starlike given by the smallest partition of $\Omega_{n-1,r+1}$. Then $$\lambda_1(T) < \lambda_1(T^\prime).$$
\end{prop}
\begin{proof}
Fixed $n$, the number of vertices and $3 \leq r \leq n-1$, the number of paths of the starlike, Theorem \ref{transitions} states that the partition giving the largest index in $\Omega_{n-1,r}$ is the balanced one. That means (essentially) that each path has the same number of vertices. More precisely, each path has at least $m=\lfloor \frac{n-1}{r}\rfloor$ vertices. The remaining $l=n-1-r*m$ vertices are distributed, one by one, in the paths from right to left. Its corresponding partition has the following format. $$T=[\underbrace{m,\ldots,m}_{r-l},\underbrace{m+1,\ldots,m+1}_{l}].$$

Theorem \ref{transitions} also states that the smallest configuration $T^\prime$ in $\Omega_{n-1,r+1}$  is given by
$$[\underbrace{1,\ldots,1}_{r},n-1-r].$$

In order to prove that $\lambda_1(T) < \lambda_1(T^\prime)$, we are going to prove something stronger: We are going to use an intermediate starlike that is path-regular: all paths have the same number of vertices, which is $m=\lceil \frac{n-1}{r} \rceil$. We observe that this starlike has (possibly) more than $n$ vertices, but $T$ is a subgraph of it. By Lemma \ref{sub}, its index is at least as large. We will show that this path-regular starlike has smaller index than next starlike with $r+1$ paths and $n$ vertices. More generally, we show
{\footnotesize $$ [\underbrace{m,\ldots,m}_{r-l},\underbrace{m+1,\ldots,m+1}_{l}]\prec
[\underbrace{n-1-r,\ldots,n-1-r}_{r}] \prec [\underbrace{1,\ldots,1}_{r},n-1-r].$$}

The first inequality follows from Lemma \ref{sub}. Given the starlike $T^{\prime\prime}$ with partition $[n-1-r,\ldots,n-1-r]$, by Theorem \ref{ref} we know that there is a $\lambda$ such that $rb_{n-1-r}=\lambda$, where $\lambda=-\lambda_1(T^{\prime\prime})$. We want to show that the starlike  $T^\prime=[1,\ldots,1,n-1-r]$ has a larger index than that of $T^{\prime\prime}$. Following Theorem \ref{maintool}, this means to show
$$ \lambda - rb_1 -b_{n-1-r} >0.$$
But this is equivalent to
$$ \lambda -r \frac{1}{\lambda} - \frac{\lambda}{r} > 0$$
or $$\frac{r\lambda^2-r^2-\lambda^2}{r\lambda}>0$$
or, since $\lambda <0$
$$r\lambda^2-r^2-\lambda^2 = \lambda^2(r-1)-r^2 < 0.$$

But this means $\lambda > - \frac{r}{\sqrt{r-1}}$, which is exactly the condition that holds for $\lambda$ by Lemma \ref{cotamagica}, so the result follows.
\end{proof}
\begin{exemplo}
 With $n=14$ vertices and $r=3$ paths, the starlike having largest index is given by the partition $[4,4,5]$. Its successor is the starlike $[1,1,1,10]$. Since $n-1-r= 13-3=10$, the intermediate starlike is $[10,10,10]$ and we have $[4,4,5] \prec [10,10,10] \prec [1,1,1,10]$.
\end{exemplo}
Following \cite{Lin2006} we have the following operation.
\begin{lema}\label{alpha}
\cite{Cve97,Lin2006}. Let $u$ be a vertex of a non-trivial connected graph
$G$, and let $G^0_{k,l}$ denote the graph obtained from $G$ by adding pendant
paths of length $k$ and $l$ at $u$. If $k \geq l \geq 1$, then
$\lambda_1(G^0_{k,l})> \lambda_1(G^0_{k+1,l-1})$.
\end{lema}

The transformation in Lemma \ref{alpha} from $G^0_{k,l}$ to $G^0_{k+1,l-1}$ is called the $\alpha_0$ transformation of $G^0_{k,l}$.

Consider a starlike $[y_1,y_2,\ldots,y_r]$. Since $y_i \leq y_{j}$, for all $i < j$, we consider applying an $\alpha_0$ transformation to $[y_1,\ldots, y_i,\ldots, y_{j},\ldots,y_r]$ obtaining $[y_1,\ldots, y_{i}+1,\ldots, y_{j}-1,\ldots,y_r]$. That is, we make a longer path shorter (by one) and a shorter path even shorter by one. According to Lemma \ref{alpha}, the index increases. We notice that this operations corresponds to the $\alpha$ operation defined in section \ref{sec:order} and state the following result for reference.

\begin{lema}\label{ref} Let $0< y_1 \leq \cdots \leq y_r $ be integers and $[y_1,\ldots,y_r]$ be a starlike tree. Then
$$\lambda_1([y_1,\ldots, y_i,\ldots,y_{j},\ldots,y_r]) < \lambda_1([y_1,\ldots, y_{i}+1,\ldots, y_{j}-1,\ldots,y_r]).$$
\end{lema}

\begin{prop}[Monotonicity inside maximal classes] \label{mono_alpha}
  Let $[y_1, \ldots, y_r]$ be a partition in a maximal class of $\Omega_{n-1,r}$ then
  $$\lambda_1(\alpha_{r-1, r}^{i-1}([y_1, \ldots, y_r])) < \cdots < \lambda_1(\alpha_{r-1, r}^{i}([y_1, \ldots, y_r])),$$
  for $i=1,\ldots,t$, where $t$ is the biggest power in the orbit.
\end{prop}
\begin{proof}
  From Lemma~\ref{ref} we can easily conclude that
  $$\lambda_1(\alpha_{r-1, r}^{i-1}([y_1, \ldots, y_r]))  \prec \lambda_1(\alpha_{r-1, r}^{i}([y_1, \ldots, y_r])),$$
  as a particular case, so the proof is done.
\end{proof}

In order to prove the next key proposition we need a technical result regarding the index  of trees associated to a certain pair of partitions.

\begin{lema}\label{inserted partition index} Let $T'=[m_1,\ldots,m_t, k,\theta,\ldots,\theta]  \in \Omega_{n'-1,r}$ and $T''= [m_1,\ldots, m_t, k + 1,  k + 1 , \ldots, k + 1, \theta] \in \Omega_{n-1,r}$ where $n'\geq n$ then $\lambda_1(T')< \lambda_1(T'').$
\end{lema}
\begin{proof}
  To see that, we recall that $\lambda=-\lambda_1(T')$ and consider the only two possible cases for $T'$: $t=0$ and $t>0$.\\

\noindent \textbf{\underline{The case $t = 0$.} }\\
We have $T' \prec T''$, where
$$T'=[k,  \theta, \ldots, \theta, \theta] \text{ and } T''= [k+1,  k+1, \ldots, k+1, \theta].$$
Let the sum of $b_i$ over $T'$ be $\lambda$, that is, $\lambda= b_{k} + (r-1) b_{\theta}=-\lambda_1([k,  \theta, \ldots, \theta, \theta])=-\lambda_1(T^\prime)$, by Theorem \ref{formulacao}.  The sum of $b_i$ over $T''$, is given by $(r-1) b_{k+1} +  b_{\theta}$. We need to prove that $\lambda_1(T')< \lambda_1(T'')$ and by Theorem~\ref{maintool}, this is equivalent to show
$b_{k} + (r-1) b_{\theta} - ((r-1) b_{k+1} +  b_{\theta}) >0,$
or $$(r-1) b_{k+1} +  b_{\theta} < \lambda.$$
We notice that, from the above equation,  $b_{\theta}= \frac{\lambda- b_{k}}{r-1}$ thus our inequality is equivalent to
$$(r-1) b_{k+1} +  \frac{\lambda- b_{k}}{r-1} < \lambda.$$
Also, $b_{k+1}=\frac{1}{\lambda -b_{k}}$ so
$\displaystyle (r-1) \frac{1}{\lambda -b_{k}} +  \frac{\lambda- b_{k}}{r-1} < \lambda$
is the inequality that we should prove. Using the fact that $b_{\theta}=\frac{\lambda- b_{k}}{r-1}$ we obtain
$$(r-1) b_{k+1} +  b_{\theta} < \lambda  \Leftrightarrow b_{\theta} + \frac{1}{b_{\theta}} < \lambda.$$

Finally, we recall that $b_{\theta} + \frac{1}{ b_{\theta}}< \lambda $ from Lemma~\ref{b_n properties}, part (b), which proves our claim.\\

\noindent \textbf{\underline{The case $t > 0$.} }\\ We have $T' \prec T''$, where
$$T'=[m_1,\ldots,m_t, k,  \underbrace{\theta, \ldots, \theta, \theta}_{s}]  \text{ and } T''= [m_1,\ldots,m_t,\underbrace{k+1,  k+1, \ldots, k+1}_{s}, \theta].$$

Define $s=r-t-1$ and $\lambda=\Sigma + b_{k} + s b_{\theta}=-\lambda_1([m_1,\ldots,m_t, k,  \theta, \ldots, \theta, \theta])$, by Theorem \ref{formulacao}. Considering the sum of $b_i$ over $T''$, given by $\Sigma+ s b_{k+1} +  b_{\theta}$, where $\Sigma= b_{1} +\cdots+  b_{t}$. If $\Sigma + s b_{k+1} +  b_{\theta} < \lambda= \Sigma+b_{k} + s b_{\theta}$  we obtain
$$\Sigma + b_{k} + s b_{\theta} - (\Sigma+ s b_{k+1} +  b_{\theta}) >0$$
thus $\lambda_1(T')< \lambda_1(T'')$, by Theorem~\ref{maintool}.

One more time we need to prove the inequality $\Sigma+ s b_{k+1} +  b_{\theta} < \lambda$. To do that we can isolate
$ \displaystyle s= \frac{\lambda - \Sigma -  b_{k}}{ b_{\theta}}$
so
$$\Sigma +s b_{k+1} + b_{\theta}= \Sigma +\left(\frac{\lambda - \Sigma -  b_{k}}{ b_{\theta}}\right) b_{k+1} + b_{\theta}=$$
$$= \Sigma +\left(\frac{\lambda - \Sigma -  b_{k}}{ b_{\theta}}\right) b_{k+1} + b_{\theta}=\Sigma +\left(\frac{(\lambda -  b_{k})b_{k+1}- \Sigma b_{k+1}}{ b_{\theta}}\right)  + b_{\theta}=$$
$$=\Sigma +\frac{1}{ b_{\theta}}-\Sigma \frac{ b_{k+1}}{ b_{\theta}}  + b_{\theta}= \left( 1 - \frac{ b_{k+1}}{ b_{\theta}}\right) \Sigma + b_{\theta} + \frac{1}{ b_{\theta}}< \lambda,$$
because $0< \frac{ b_{k+1}}{ b_{\theta}}<1$, $(\lambda -  b_{k})b_{k+1}=1$ and $b_{\theta} + \frac{1}{ b_{\theta}}< \lambda $ from Lemma~\ref{b_n properties}, part b).
\end{proof}

\begin{prop}[Monotonicity between classes] \label{mon_complex}
  Let $T=[y_1, \ldots, y_r] \in \Omega_{n-1,r}$ be the largest partition in a maximal class and $T''=[z_1, \ldots, z_r] \in \Omega_{n-1,r}$ be the smallest partition of the next maximal class, then
  $$\lambda_1(T)< \lambda_1(T'').$$
\end{prop}
\begin{proof}

According to Theorem~\ref{transitions} these are the only possibilities, that is, if the covering is not Type I or Type II then
$T \rightarrow T''$
implies that \begin{itemize}
\item[(i)] either $T''=  \alpha_{t+1, r}(T)$ or
\item[(ii)]$T= [m_1,\ldots,m_t, k,\ldots,\xi]$,\\
       $T''= [m_1,\ldots,m_t, k+1,\ldots,\theta]$, with $\theta \geq \xi$.
\end{itemize}

\textbf{Case (i):} we can use Lemma~\ref{ref} to conclude that $\lambda_1(T)< \lambda_1(T'')$ (for example $[1, 3, 4, 5]\rightarrow[1, 4, 4, 4]$).\\

\textbf{Case (ii):} the property $\theta \geq \xi$ alows us to insert an intermediary partition (with possibly bigger sum) $$T'=[m_1,\ldots,m_t, k,\theta,\ldots,\theta]$$
between $T$ and $T''$ (for example $[1, 2, 5, 5]\rightarrow [1, 3, 3, 6]$ or $[1, 4, 4, 4]\rightarrow [2, 2, 2, 7]$).

By construction $\theta \geq \xi$ so the tree $T$ is a subgraph of the tree $T'$. From Lemma~\ref{sub} we conclude that $\lambda_1(T)< \lambda_1(T')$.
Finally, from Lemma~\ref{inserted partition index} we conclude that $\lambda_1(T')< \lambda_1(T'')$ completing our proof.
\end{proof}

\begin{exemplo}
  For $\Omega_{13,4}$ the typical changes between maximal classes will be something like\\
\begin{center}
\begin{minipage}[l]{5.5cm}
$\,\,T=[1, 2, 5, 5]$\\
{\color{blue}$ \,T'=[1, 2, 6, 6]\in \Omega_{15,4}$ } \\
$T''=[1, 3, 3, 6]$\\
Type III
\end{minipage}
\begin{minipage}[l]{5.5cm}
$\,\,T=[1, 4, 4, 4]$\\
{\color{blue}$\,T'=[1, 7, 7, 7] \in \Omega_{22,4}$ }\\
$T''=[2, 2, 2, 7]$\\
Type III
\end{minipage}
\begin{minipage}[l]{4cm}
$\,\,T=[1, 3, 4, 5]$\\
$ ^{ }$\\
$T''=[1, 4, 4, 4]$\\
Type III / $\alpha$
\end{minipage}
\end{center}

\end{exemplo}

\section*{Acknowledgments}
Vilmar Trevisan acknowledges the partial support of CNPq - Grants 409746/2016-9 and 303334/2016-9.
\bibliographystyle{plain}

\bibliography{refs}

\end{document}